\documentclass[12 pt,oneside,reqno,a4paper]{amsart} 
\usepackage{amsfonts,amssymb,amscd,amsmath,enumerate,verbatim,calc} 
\usepackage{float}
\usepackage{amsthm}
\newtheorem{theorem}{Theorem}[section]
\newtheorem{definition}{Definition}[section]

\newtheorem{lemma}[theorem]{Lemma}

\newtheorem{proposition}[theorem]{Proposition}
\newtheorem{remark}[theorem]{Remark}
\usepackage{mathtools}

\numberwithin{equation}{section}

\usepackage{amsfonts}

\begin{document}
	
		\title{ Multiplicative Lie algebra structure on nilpotent group of class $2$}
	\author{Deepak Pal$^{1}$, Amit Kumar$^{2}$ and Sumit Kumar Upadhyay$^{3}$ \\
		{Department of Applied Sciences,\\ Indian Institute of Information Technology Allahabad\\Prayagraj, U. P., India} }	

	\thanks{$^1$deepakpal5797@gmail.com, $^2$amitiiit007@gmail.com, $^3$upadhyaysumit365@gmail.com}

	\begin{abstract}  
		This paper explores the properties of multiplicative Lie algebra structures on a nilpotent group of class $2$. We also present a method for determining a multiplicative Lie algebra structure on a group that serves as an extension of one Lie ring by another  Lie ring such as a metacyclic group and a nilpotent group of class $2$.	
		
	\end{abstract}
	\maketitle 
	
	\section{Introduction}
A group $G$ is nilpotent of class $2$ if $[G, G] \subseteq Z(G)$, where $Z(G)$ and $[G, G]$ denote the center and commutator subgroup of $G$, respectively. One of the most famous examples of a nilpotent group of class $2$ is the Heisenberg group, which is important in both quantum mechanics and number theory. Nilpotent groups of class 2 form an important subclass of finite $p-$groups. Understanding the structure and properties of these groups helps in the broader study of finite groups and their classification.

In 1993, Ellis \cite{GJ} proposed a conjecture suggesting that every universal commutator identity can be derived from the five well-known universal commutator identities. He offered a partial confirmation of this conjecture by employing homological techniques and introducing a new algebraic structure termed ``Multiplicative Lie algebra". A multiplicative Lie algebra is a triple $ (G,\cdot,\star) $,  where $ (G,\cdot) $ is a group  and $\star$ is a binary operation on $G$ such that for all  $x, y, z  \in  G$, the following identities hold:
\begin{enumerate}
	\item $ x\star x=1 $, 
	\item $ x\star(yz)=(x\star y){^y(x\star z)} $, 
	\item $ (x y)\star z= {^x(y\star z)} (x\star z) $,
	\item $ ((x\star y)\star{^yz})((y\star z)\star{^zx})((z\star x)\star {^xy})=1 $,
	\item $ ^z(x\star y)=({^zx}\star {^zy})$, 
\end{enumerate}
where $^zx$ denotes the element $zxz^{-1}$. In any group $G$, there always exists a trivial multiplicative Lie algebra structure, given by $x\star y =1$ for all $x,y \in G$. However, if $G$ is a non-abelian group, there are at least two distinct multiplicative Lie algebra structures on $G$ : the trivial one, and another (called improper multiplicative Lie algebra structure) defined by $x\star y =[x,y]$ for all $x,y \in G$. Therefore, it becomes an interesting problem to classify all possible multiplicative Lie algebra structures on a given group $G$.
The work accomplished in this direction is detailed in \cite{AD,RLS,MS,GLW}. Expanding on the concept of multiplicative Lie algebra, Point and Wantiez \cite{FP} further developed the notions of solvability and nilpotency. In relation  to improper multiplicative Lie algebras, the authors \cite{ML} introduced the concepts of Lie commutator and multiplicative Lie center, leading to a new understanding of solvable and nilpotent multiplicative Lie algebras.

In this paper, we conclusively show that for nilpotent groups of class $2$, the concept of nilpotency for multiplicative Lie algebras as introduced in \cite{FP} directly implies the Lie nilpotency described in \cite{ML}. Additionally, we give a new multiplicative Lie algebra structure on a group by combining two existing multiplicative Lie algebra structures on that group under certain conditions. We examine essential properties of multiplicative Lie algebra structures within nilpotent groups of class $2$. Moreover, we present a method to compute a multiplicative Lie algebra structure on a given group $G$, which is an extension of a Lie ring by a Lie ring, specifically for nilpotent groups of class $2$, metacyclic groups, and two-generator nilpotent groups of class $2$ with order $p^n$, where $p$ is a prime.

	\section{Preliminaries} 
We start the section with definitions and results for multiplicative Lie algebra, which will be used throughout the article.

	\begin{proposition}\cite{GJ} \label{MLA}
		Let $(G,\cdot,\star)$ be a  multiplicative Lie algebra. Then the following identities hold:
		\begin{enumerate}
			\item $ 1\star x = x\star 1 = 1$
			
			\item $ (x\star y)(y\star x) = 1 $
			
			\item $ ^{(x\star y)}(u\star v) = {^{[x,y]}(u\star v)} $  
			
			\item $ [(x\star y),z] = [x,y]\star z$
			
			\item $ x^{-1}\star y = {^{x^{-1}}(x\star y)^{-1}}$ and $ x\star y^{-1} = {^{y^{-1}}(x\star y)^{-1}}$
			
		\end{enumerate}
		for all $x,y, z, u,v \in G.$
	\end{proposition}
	
		\begin{definition}\cite{FP} Let $G$ be a multiplicative Lie algebra. Then
		
		\begin{enumerate}
		\item We call $G$ a Lie ring if $G$ is an abelian group.
			\item $G$ is said to be solvable if $\Gamma^{(n)}(G) = 1$
			for some $n$, where $\Gamma^{(0)}(G) = G,$ and $\Gamma^{(k+1)}(G) = \Gamma^{(k)}(G) \star \Gamma^{(k)}(G)$ for all $k\in \mathbb{N}\cup \{0\}$. We say that $G$ is solvable of class $n$ if $\Gamma^{(n)}(G) = 1$ and $\Gamma^{(n-1)}(G) \neq 1.$ 
			
			\item $G$ is said to be nilpotent if $\Gamma_{(n)}(G) = {1}$
			for some $n$, where $\Gamma_{(1)}(G) = G$, and $\Gamma_{(k+1)}(G) = G \star \Gamma_{(k)}(G)$ for all $k\in \mathbb{N}$. The smallest $n$ such that  $\Gamma_{(n+1)}(G) = {1}$ is called the nilpotency class of $G.$  
			
		\end{enumerate}
	\end{definition}
	
	\begin{definition} \cite{ML} Let $G$ be a multiplicative Lie algebra. The ideal $^L[G, G]$ generated by the set $\{(a\star b)^{-1}[a, b]\mid a, b \in G\}$ of $G$ is called  the Lie commutator of $G.$ The element $(a\star b)^{-1}[a, b]$ is denoted by $^L[a, b]$.
	\end{definition}

	\begin{proposition}\cite{ML} \label{Lie}
		Let $G$ be a multiplicative Lie algebra.  Then the following identities hold:
	
	\begin{enumerate}
		\item  $^L[a, a] = 1$ 
		
		\item  $^L[a, b] \ ^L[b, a] = 1$
		
		\item $^L[ab, c] = {^L[a, c]} \ ^{(^c{a})} {(^L[b, c])}$ 
		
		\item $ ^L[a, bc] = {^{b}{(^L[a, c])}} \  {^{[^{b}c, ^{b}a]}({^L[a, b]})}$ 
		
		\item $ ^a{(^L[b, c])} = {^L[^ab, ^ac]}$ 
		
		\item $^L[a^{-1}, b] = {^{a^{-1}}{(^L[b, a])}} \ \text{and} \ ^L[a, b^{-1}] = {^{b^{-1}}{(^L[b, a])}}$ 
		
		\item $^{^L[a, b]}{(x \star y)} = {(x \star y)}. $ In particular, $[^L[a, b], {(x \star y)}] = 1$
	\end{enumerate}
	for all $a, b, c, x, y \in G.$
\end{proposition}
	
	\begin{definition}\cite{ML} Let $G$ be a multiplicative Lie algebra. Then
	
	\begin{enumerate}
	\item $G$ is said to be Lie solvable if $G^{(n)} = 1$
for some $n$, where $G^{(0)} = G,$ and $G^{(k+1)} = {^L[G^{(k)}, G^{(k)}]}$ for all $k\in \mathbb{N}\cup \{0\}$. The smallest $n$ such that  $G^{(n)} = 1$ is called the derived length of $G.$

\item $G$ is said to be Lie nilpotent if $L_n(G) = {1}$
for some $n$, where $G =L_0(G)$, and $L_{k+1}(G) =  {^L[G, L_k(G)]}$ for all $k\in \mathbb{N}\cup \{0\}$. We say that $G$ is nilpotent of class $n$ if $L_n(G) = {1}$ but $L_{n-1}(G)  \neq 1.$ 
	
	\end{enumerate}
\end{definition}

\begin{proposition} \cite{ML}
	The set $MZ(G) = \{a \in G : \ ^L[a, b] = 1, \forall \ b \in G \}$ is an ideal
	of $G,$ called as the multiplicative Lie center of $G.$ 
\end{proposition}

\section{Some properties of nilpotent groups of class $2$}
In this section, we study some properties of multiplicative Lie algebra structures on nilpotent group of class $2$. 

\begin{proposition} \label{C1}
	Let $G$ be a nilpotent group of class $2$ with multiplicative Lie algebra structure $\star.$ Then 
	\begin{enumerate}
		\item There is only trivial multiplicative Lie algebra structure on the commutator subgroup $[G, G]$ . 
		\item  The subgroup $G \star G$ of $G$, generated by the set $\{ x \star y \mid x, y\in G \}$,  is abelian.
		
		\item  The subgroup  $^L[G, G]$ of $G$, generated by the set $\{(x\star y)^{-1}[x, y]\mid x, y \in G\}$,  is abelian.
		
		\item  $ MZ(G) \star MZ(G) \subseteq Z(G).$
		
		\item $(G\star G)\star [G, G] = 1.$    
	\end{enumerate}
\end{proposition}

\begin{proof} \textit{(1)} Let $[x, y], [u,v]\in [G,G].$ Then by Proposition \ref{MLA}, we have  $[x, y]\star [u,v] = [(x\star y), [u,v]] = 1 $ as $ [G,G]\subseteq Z(G)$. Hence, $g\star g'= 1$, for all $g, g'\in [G, G].$
	
\textit{(2)} Let $(x\star y), (u \star v )\in G\star G.$ Then  by Proposition \ref{MLA}, we have  $ ^{(x\star y)}(u\star v) = {^{[x,y]}(u\star v)}.$  This implies that $ {(x\star y)}(u\star v) = (u\star v)(x\star y)$ as $G$ is a nilpotent group of class $2.$ Hence, $G \star G $ is abelian.
	
\textit{(3)} To prove $^L[G, G]$ is abelian, it suffices to show that $[^L[x, y], ^L[u, v] ] = 1$ for all  $^L[x, y], ^L[u, v] \in {^L[G, G]}.$ Now, we have
\begin{align*}
		[^L[x, y], ^L[u, v] ] &= [(x\star y)^{-1}[x, y], ^L[u, v] ] \\
		&=  {^{(x\star y)^{-1}}[[x, y], ^L[u, v] ]} [(x\star y)^{-1}, ^L[u, v] ] \\&=  {^{(x\star y)^{-1}}[[x, y], ^L[u, v] ]} [(x\star y)^{-1}, (u\star v)^{-1} ] ~{^{(u\star v)^{-1}}[(x\star y)^{-1}, [u, v] ]} \\
		&= 1~ (\text{since ~} [G, G] \subseteq Z(G) ~\text{and ~}  G \star G ~\text{is~ abelian}).
\end{align*}
	Hence, we are done.
	
\textit{(4)}  To prove  $ MZ(G) \star MZ(G)  \subseteq  Z(G),$ it suffices to show that $[z, x\star y] = 1$ for every $x, y \in MZ(G) \ \text{and} \ z\in G.$	Now, by the definition of multiplicative Lie center and Proposition \ref{MLA}, we have
\begin{align*}
	^L[[x, y], z] &= 1  \\& \hspace{-2.7cm} \implies ([x, y]\star z)^{-1} [[x, y], z] = 1\\& \hspace{-2.7cm} \implies [z, x\star y] = 1
\end{align*}
 Hence,  $ MZ(G) \star MZ(G)  \subseteq  Z(G).$
	
\textit{(5)} Let $(x\star y)\star [u, v]\in (G\star G)\star [G, G]. $ Then by Proposition \ref{MLA}, we have $(x\star y)\star [u, v] = [(x\star y), (u\star v)] = 1.$ Hence $(G\star G)\star [G, G] = 1 .$  
\end{proof}

\begin{lemma}\label{LZG}
	Let $G$ be a nilpotent group of class $2$ with multiplicative Lie algebra structure $\star$ such that the induced multiplicative Lie algebra structure on $\frac{G}{[G, G]}$ is trivial. Then $[G, G] \subseteq LZ(G).$ 
\end{lemma}
\begin{proof}
	Without loss of generality, assume $[x, y] \in[G, G]$ and $g \in G.$ Then by Proposition \ref{MLA}, we have $[x, y] \star g = [x \star y, g].$ Thus, to prove $[G, G] \subseteq LZ(G),$ it is enough to show that $G\star G \subseteq Z(G).$ Since $\frac{G}{[G, G]}$ has trivial  multiplicative Lie algebra structure,  $a[G, G] \star b[G, G] = (a \star b)[G, G] = [G, G]$ for all $a, b \in G$. Hence,  $(a \star b) \in [G, G] \subseteq Z(G).$ Hence $G\star G \subseteq Z(G).$   This completes the proof.   
\end{proof}

\begin{lemma}\label{Lemma}
Let $G$ be a nilpotent group of class $2$ with multiplicative Lie algebra structure $\star.$ Then
$[G, L_n(G)] \subseteq \Gamma_{(n+1)}(G),  \ \text{for all} \ n \in \mathbb{N}.$ 
\end{lemma}
\begin{proof}
We prove the lemma by induction on $n$. Let $[g, ^L[x, y]] \in [G, L_1(G)], $ where $g\in G \ \text{and} \ ^L[x, y] \in L_1(G).$ Then $[g, ^L[x, y]] = [g, (x \star y)^{-1}[x, y]] = [g, (x \star y)^{-1}] = {^g{(x \star y)^{-1}}} (x \star y) \in G \star G = \Gamma_{(2)}(G).$ Hence $[G, L_1(G)] \subseteq \Gamma_{(2)}(G). $ Thus, the result holds for $n = 1.$

Assume the result holds for $n$. Let  $[g, ^L[x, y ]] \in [G, ^L[G, L_n(G) ]] = [G, L_{n+1}(G) ]$, where $y \in L_n(G)$. Then by Proposition \ref{MLA}, we have $[g, ^L[x, y ]] =  [g,  ( x\star y)^{-1} [x ,y]] = [g,  y\star x] = g \star [y, x] .$ Now, by the induction hypothesis, we have $[x, y]^{-1} = [y, x] \in \Gamma_{(n+1)}(G),$ so $g \star [y, x] \in \Gamma_{(n+2)}(G).$ Thus $[G, L_{n+1}(G)] \subseteq \Gamma_{(n+2)}(G).$ 	Hence, the result follows.   
\end{proof}

\begin{theorem} \label{KC}
	Let $G$ be a nilpotent group of class $2$ with multiplicative Lie algebra structure $\star.$ If $G$ is a nilpotent (solvable) multiplicative Lie algebra, then $G$ is also Lie nilpotent (solvable). 
\end{theorem}
\begin{proof}
Suppose $G$ is nilpotent multiplicative Lie algebra of class $n-1.$ Then $\Gamma_{(n)}(G) = {1}.$ To prove the theorem, it is sufficient to show that $ L_n(G) \subseteq \Gamma_{(n)}(G)$ $\text{for all} \ n \in \mathbb{N}.$ We proceed with the proof of the claim by induction on $n.$ Since $L_1(G)$  is a subgroup of $G$. So, $L_1(G) \subseteq G.$ Thus, the result holds for $n = 1.$ 

Suppose the statement is true for $n-1$. Let ${^L[g, h]} \in {^L[G, L_{n-1}(G)]} = L_n(G),$ where $h \in L_{n-1}(G).$ Then $ ^L[g, h] =  (g \star h )^{-1} [g, h].$ Now, by the induction hypothesis and Lemma \ref{Lemma}, we have $ L_n(G) \subseteq \Gamma_{(n)}(G) \ \text{for all} \ n \in \mathbb{N}.$ However, by assumption, we have $\Gamma_{(n)}(G) = {1},$ so $L_n(G) = 1.$ We conclude that $G$ is Lie nilpotent.    

Similarly, we can prove the solvable case.
\end{proof}

In the next theorem, we construct a new multiplicative Lie algebra structure with the help of two given multiplicative Lie algebra structures under certain conditions. 
 \begin{theorem} \label{C5}
 Let $G$ be a group with  multiplicative Lie algebra structures $\star_1$ and  $\star_2$, satisfying the following conditions for all $x, y, z, w \in G$:
 \begin{enumerate}
 	\item $ G\star_1 G $ and $ G\star_2 G $ are abelian;
 	\item $ (x\star_1 y)(z \star_2 w)=(z \star_2 w)(x\star_1 y) $; 
 	\item  $ ((x\star_1 y){\star_2} {^yz})((x\star_2 y){\star_1} {^yz})((y\star_1 z)\star_2{^zx})((y\star_2 z)\star_1{^zx})((z\star_1 x)\star_2 {^xy})((z\star_2 x)\star_1 {^xy})=1 $.
 \end{enumerate}
 Then the binary operation $\star $ on $G,$ given by $x \star y = (x\star_1 y)(x\star_2 y) $ defines a multiplicative Lie algebra structure on $G.$  
 \end{theorem}

\begin{proof} Let $x, y, z \in G.$  Then
\begin{enumerate}
	
		\item $x \star x = (x\star_1 x)(x\star_2 x) = 1$.
		\item  $ x\star(y \cdot z)=(x\star_1 (y \cdot z))(x\star_2 (y \cdot z)) = (x\star_1 y) ^y(x\star_1 z) (x\star_2 y) ^y(x\star_2 z) = (x\star_1 y) (x\star_2 y) ^y(x\star_1 z) ^y(x\star_2 z) = (x\star y)\cdot{^y(x\star z)} $.
		\item $ (x \cdot y)\star z= ((x \cdot y)\star_1 z)  ((x \cdot y)\star_2 z) = {^x( y\star_1 z)} ( x\star_1 z)  {^x( y\star_2 z)} ( x\star_2 z) = {^x( y\star_1 z)} {^x( y\star_2 z)} ( x\star_1 z)   ( x\star_2 z) = {^x(y\star z)} \cdot (x\star z)$.
		\item We have
		\begin{align*}
		((x\star y)\star{^yz}) & = ((x\star_1 y) (x\star_2 y) \star {^yz}) \\& = ((x\star_1 y) (x\star_2 y) \star_1 {^yz}) ((x\star_1 y) (x\star_2 y) \star_2 {^yz}) \\ & = ( (x\star_2 y)(x\star_1 y) \star_1 {^yz}) ((x\star_1 y) (x\star_2 y) \star_2 {^yz}) \\& = {^{(x\star_2 y)}}((x\star_1 y) \star_1 {^yz}) ((x\star_2 y) \star_1 {^yz}) {^{(x\star_1 y)}}( (x\star_2 y) \star_2 {^yz}) ( (x\star_1 y) \star_2 {^yz}) \\& = ((x\star_1 y)\star_1{^yz}) ((x\star_2 y)\star_1{^yz})((x\star_2 y)\star_2{^yz}) ((x\star_1 y)\star_2{^yz}).
		\end{align*}
	Hence,   
		\begin{align*}
			((x\star y)\star{^yz})((y\star z)\star{^zx})((z\star x)\star {^xy}) & = ((x\star_1 y)\star_1{^yz})  ((x\star_2 y)\star_1{^yz}) ((x\star_2 y)\star_2{^yz}) \\& \hspace{-4.3cm} ((x\star_1 y)\star_2{^yz}) ((y\star_1 z)\star_1{^zx}) ((y\star_2 z)\star_1{^zx})((y\star_2 z)\star_2{^zx}) ((y\star_1 z)\star_2{^zx})  \\& \hspace{-4.3cm}    ((z\star_1 x)\star_1{^xy}) ((z\star_2 x)\star_1{^xy})((z\star_2 x)\star_2{^xy}) ((z\star_1 x)\star_2{^xy}) \\&  \hspace{-4.8cm} = ((x\star_1 y)\star_1{^yz})((y\star_1 z)\star_1{^zx}) ((z\star_1 x)\star_1{^xy})    ((x\star_2 y)\star_2{^yz}) ((y\star_2 z)\star_2{^zx}) \\& \hspace{-4.3cm} ((z\star_2 x)\star_2{^xy}) ((x\star_1 y)\star_2{^yz}) ((x\star_2 y)\star_1{^yz})((y\star_2 z)\star_1{^zx}) ((y\star_1 z)\star_2{^zx})     \\&  \hspace{-4.3cm} ((z\star_2 x)\star_1{^xy})  ((z\star_1 x)\star_2{^xy}) \\&  \hspace{-4.8cm} =  ((x\star_1 y)\star_2{^yz}) ((x\star_2 y)\star_1{^yz})((y\star_2 z)\star_1{^zx}) ((y\star_1 z)\star_2{^zx})      ((z\star_2 x)\star_1{^xy}) \\&  \hspace{-4.3cm}  ((z\star_1 x)\star_2{^xy}) \\&  \hspace{-4.8cm} = 1.
		\end{align*}
		\item $ ^z(x\star y) = {^z((x\star_1 y) (x\star_2 y))} = {(^zx\star_1 {^zy}) (^zx\star_2 {^zy})} =(^zx\star ^zy).$  
\end{enumerate} 
This shows that $\star$ is a multiplicative Lie algebra structure on $G.$
\end{proof}

\begin{lemma} \label{C6}
Let $ (G,\cdot,\star_1)$ be a multiplicative Lie algebra such that $[G, G]\subseteq LZ(G).$ Then for any other multiplicative Lie algebra structure $\star_2$ on $G,$ we have $[(x\star_1 y), (z \star_2 w)] = 1 $ for all  $ x, y, z, w \in G .$  
\end{lemma}
\begin{proof}
Let $x\star_1 y \in G\star_1 G$ and  $z\star_2 w \in G\star_2 G.$ Then by Proposition \ref{MLA}, we have $[(x\star_1 y), (z \star_2 w)] = [x, y] \star_1 (z \star_2 w).$ Since $[G, G]\subseteq LZ(G),$ $[(x\star_1 y), (z \star_2 w)] = 1.$   	
\end{proof} 

\begin{theorem}
Let $G$ be a nilpotent group of class $2.$ Suppose $\star_1$ is a multiplicative Lie algebra structure on $G$ such that $[G, G]\subseteq LZ(G).$ Then, for any other  multiplicative Lie algebra structure $\star_2$ on $G,$ the binary operation $\star$ defined as  $x \star y = (x\star_1 y)(x\star_2 y) $ is a multiplicative Lie algebra structure on $G,$ provided  $ ((x\star_1 y){\star_2} {^yz})((x\star_2 y){\star_1} {^yz})((y\star_1 z)\star_2{^zx})((y\star_2 z)\star_1{^zx})((z\star_1 x)\star_2 {^xy})((z\star_2 x)\star_1 {^xy})=1 $.   
\end{theorem}
\begin{proof}
This follows from Proposition \ref{C1}, Theorem \ref{C5},  and Lemma \ref{C6}. 
\end{proof}

\section{Extension of a lie ring by a lie ring}	
In this section, we provide a method to compute a multiplicative Lie algebra structure on a given group $G$ which is  an extension of a Lie ring $H$ by a Lie ring $K$.	

Consider an extension $ \mathcal{E}(H, K):	{1}\longrightarrow  H\overset{i} \longrightarrow G\overset{\beta} \longrightarrow  K \longrightarrow {1}$ of $H$ by $K,$  where $H$  and $K$ are Lie rings. Let $t: K \to G $ be a section. Then every element of the group $G$ can be uniquely expressed in the form $ht(x)$ for some $h\in H$ and $x\in K.$ Then by (\cite{RL}), the group operation $``\cdot"$  in $G$ is given by
\begin{align*}
	ht(x)\cdot kt(y) &= h\sigma_x^t(k)f^t(x,y)t(xy) \tag*{(1)}
\end{align*}
where $\sigma_x^t: H \to H $ is a group isomorphism defined by $\sigma_x^t(h)=t(x)ht(x)^{-1}$ and  $f^t: K\times K \longrightarrow H$ is a map satisfying
\begin{align*}
	f^t(1,x) = f^t(x,1) = 1, \ \text{and} \ f^t(x,y)f^t(xy, z) = \sigma_x^t(f^t(y,z))f^t(x, yz) \tag*{(2)}
\end{align*}
\begin{remark}
	The map $\sigma^t: K \to Aut(H)$ defined by $\sigma^t(x) = \sigma^t_x $ is independent of the choice of section $t$ (see \cite{RL}, p. 378). So, now onwards we denote $\sigma^t$  by $\sigma.$
\end{remark}

In the following lemma, we mention further identities satisfied by $\sigma_x$ and $f^t$.
\begin{lemma}
	To each $x\in K$, the homomorphism $ \sigma_x : H \to H$ and the map $f^t: K\times K \longrightarrow H$ also satisfy the following identities for every $x, y \in K$:
	
	\begin{enumerate}
		\item $\sigma_{xy}(f^t(x^{-1}, x)) = f^t(y, x) f^t(xy, x^{-1})$
		\item $\sigma_{y}(f^t(x, x^{-1})) = f^t(y, x) f^t(xy, x^{-1})$
		\item $\sigma_{xy}(f^t(x^{-1}, x)) = \sigma_{y}(f^t(x, x^{-1}))$
		%	\item $\sigma_{x}(f^t(x^{-1}, x)) = f^t(x, x^{-1})$
		%	\item $\sigma_{x}(f^t(y, y^{-1})) f^t(y, x) = \sigma_{y}(f^t(x, y^{-1}))f^t(y, xy^{-1})f^t(x, y).$
	\end{enumerate}
\end{lemma} 
\begin{proof} 
	\textit{(1)} Since $K$ is an abelian group,  $t(xy) = t(yx)$ for every $x, y \in K.$ Thus $f^t(x, y)^{-1} t(x)t(y) = f^t(y, x)^{-1} t(y)t(x).$ This implies that 
	\begin{align*}
		\hspace{-2cm}[t(x), t(y)] &= f^t(x, y)f^t(y, x)^{-1} \tag*{(\textit{i})} \\
		\text{Also,} \ [t(x), t(y)] = t(x) t(y)t(x)^{-1}t(y)^{-1} &= f^t(x, y) t(xy) f^t(x^{-1}, x)^{-1} t(x^{-1})t(y)^{-1} \\& = f^t(x, y) \sigma_{xy}(f^t(x^{-1}, x)^{-1})  t(xy) t(x^{-1})t(y)^{-1} \\&= f^t(x, y) \sigma_{xy}(f^t(x^{-1}, x)^{-1})  f^t(xy, x^{-1}). \tag*{(\textit{ii})} 
	\end{align*}
	Equating $(i)$ and $(ii)$, we deduce the result.
	
	\textit{(2)} In equation $(2)$ replace $x \rightarrowtail y, \ y \rightarrowtail x$ and $z \rightarrowtail x^{-1},$ we get desired result.
	
	\textit{(3)} From identities \textit{(1)} and \textit{(2)}, we can deduce \textit{(3)}.
\end{proof}
Now, $\beta(t(x)\star h)= 1 $. Thus, we have a  homomorphism $\Gamma_x^t: H \to H$ given by $\Gamma_x^t(h) = t(x)\star h$.\\
Also, $\beta((t(x)\star t(y))t(x\star y)^{-1})= 1 $. This gives us a map $h^t:K\times K \to H$ such that  $h^t(x,y)= (t(x)\star t(y))t(x\star y)^{-1}$, where $h^t$ satisfies 
\begin{align*}
	h^t(x,1)=h^t(1,x)=h^t(x,x)=1. \tag*{(3)}
\end{align*}
Since $h \star kt(y) = (h\star k){^k(h\star t(y))} = (h\star k){(h\star t(y))} = (h\star k) \Gamma^t_y(h^{-1})$  and $$ t(x) \star kt(y) = (t(x) \star k) {^k((t(x)\star t(y) )} = \Gamma^t_x(k) h^t(x, y) {^k(t(x \star y) )}.$$
Therefore, we get  the following expression:
\begin{align*}
	ht(x){\star} kt(y) & =  {^h(t(x) \star kt(y))}(h\star kt(y)) \\& = \Gamma^t_x(k) h^t(x, y) {^{hk}(t(x \star y) )} (h \star k) \Gamma^t_y(h^{-1}) \\& = hk \Gamma^t_x(k) \sigma_{(x \star y)} (h^{-1}k^{-1}(h \star k) \Gamma^t_y(h^{-1})) h^t(x, y) t(x \star y).  \tag*{(4)}
\end{align*}
 On the other hand, we have
 \begin{align*}
 	ht(x){\star} kt(y) & =  {(ht(x) \star k)}{^k(ht(x)\star t(y))} \\& =  {^h(t(x) \star k)} (h \star k) {^k(^h(t(x)\star t(y))(h \star t(y))) }  \\& = hk \Gamma^t_x(k) (h\star k)  h^t(x, y) t(x \star y) h^{-1}k^{-1} \Gamma^t_y(h^{-1})\\& = hk \Gamma^t_x(k) (h\star k)  h^t(x, y)  \sigma_{(x\star y)} (h^{-1}k^{-1} \Gamma^t_y(h^{-1})) t(x \star y).  \tag*{(5)}
 \end{align*}
 On comparing the equations $(4)$ and $(5)$, we get
 \begin{align*}
 	\sigma_{(x\star y)} (h\star k) =  h\star k. \tag*{(6)}
 \end{align*}
This gives us a multiplicative Lie algebra structure $``\star"$ in $G$ as follows: 
\begin{align*}
	ht(x){\star} kt(y) = hk \Gamma^t_x(k) (h \star k) \sigma_{(x \star y)} (h^{-1}k^{-1} \Gamma^t_y(h^{-1})) h^t(x, y) t(x \star y). \tag*{(7)}
\end{align*}
Next, we will examine some of the properties of the function $h^t$.

Consider the expression 
\begin{align*}
	(ht(x)\cdot kt(y))\star lt(z) &= h\sigma_x(k) f^t(x, y) t(xy)  \star lt(z) \\& \hspace{-2.6cm}= hl\sigma_x(k) \Gamma^t_{xy}(l) ( h\sigma_x(k) f^t(x, y) \star l) f^t(x, y) \sigma_{(xy) \star z}(h^{-1}l^{-1}\sigma_x(k^{-1}) {f^t(x, y)}^{-1}  \\& \hspace{-2.2cm} \cdot \Gamma^t_{z}(h^{-1}\sigma_x(k^{-1}) f^t(x, y)^{-1}))   h^t(xy ,z) t(xy \star z). \tag*{(8)}
\end{align*} 
On the other hand, we have
\begin{align*}
	(ht(x)\cdot kt(y))\star lt(z)&={^{ht(x)}}(kt(y)\star lt(z)) \cdot (ht(x)\star lt(z)) \\&\hspace{-3.6cm}= ht(x) (kl \Gamma^t_y(l) (k \star l) \sigma_{(y\star z)}(k^{-1}l^{-1}\Gamma^t_z(k^{-1}))h^t(y ,z )t(y \star z)) t(x)^{-1}h^{-1}  \\& \hspace{-3.2cm}  \cdot (hl\Gamma^t_x(l)(h \star l)\sigma_{(x\star z)}(h^{-1}l^{-1})\Gamma^t_z(h^{-1}))h^t(x ,z )t(x \star z))
   \\& \hspace{-3.6cm} = h\sigma_x (kl \Gamma^t_y(l) (k \star l) \sigma_{(y\star z)}(k^{-1}l^{-1}\Gamma^t_z(k^{-1}))h^t(y ,z)) f^t(x, y\star z) t(x(y\star z)) t(x^{-1}) \\& \hspace{-3.2cm}\cdot f^t(x, x^{-1})^{-1}   (l \Gamma^t_x(l) (h \star l) \sigma_{(x\star z)}(h^{-1}l^{-1}\Gamma^t_z(h^{-1})) h^t(x ,z )t(x \star z))
    \\& \hspace{-3.6cm} = h\sigma_x (kl \Gamma^t_y(l) (k \star l) \sigma_{(y\star z)}(k^{-1}l^{-1}\Gamma^t_z(k^{-1}))h^t(y ,z)) f^t(x, y\star z) f^t(x(y\star x), x^{-1}) t(y\star z)  \\& \hspace{-3.2cm}\cdot  f^t(x, x^{-1})^{-1}  (l \Gamma^t_x(l) (h \star l) \sigma_{(x\star z)}(h^{-1}l^{-1}\Gamma^t_z(h^{-1})) h^t(x ,z )t(x \star z))
     \\& \hspace{-3.6cm} = h\sigma_x (kl \Gamma^t_y(l) (k \star l) \sigma_{(y\star z)}(k^{-1}l^{-1}\Gamma^t_z(k^{-1}))h^t(y ,z)) f^t(x, y\star z) f^t(x(y\star x), x^{-1}) \\& \hspace{-3.2cm}\cdot  \sigma_{y\star z} (f^t(x, x^{-1})^{-1}) \sigma_{(y\star z)}(l \Gamma^t_x(l) (h \star l) \sigma_{(x\star z)}(h^{-1}l^{-1}\Gamma^t_z(h^{-1})) h^t(x ,z )) t(y\star z) t(x \star z)
     \\& \hspace{-3.6cm} = h\sigma_x (kl \Gamma^t_y(l) (k \star l) \sigma_{(y\star z)}(k^{-1}l^{-1}\Gamma^t_z(k^{-1}))h^t(y ,z)) f^t(x, y\star z) f^t(x(y\star x), x^{-1})  \\& \hspace{-3.2cm}\cdot  f^t(y\star z, x)^{-1} f^t(x(y\star x), x^{-1})^{-1}  \sigma_{(y\star z)}(l \Gamma^t_x(l) (h \star l) \sigma_{(x\star z)}(h^{-1}l^{-1}\Gamma^t_z(h^{-1})) h^t(x ,z )) \\& \hspace{-3.2cm}\cdot f^t(y\star z, x \star z) t((y\star z)(x \star z))
      \\& \hspace{-3.6cm} = h\sigma_x (kl \Gamma^t_y(l) (k \star l) \sigma_{(y\star z)}(k^{-1}l^{-1}\Gamma^t_z(k^{-1}))h^t(y ,z)) f^t(x, y\star z)  f^t(y\star z, x)^{-1}  \\& \hspace{-3.2cm}\cdot \sigma_{(y\star z)}(l \Gamma^t_x(l) (h \star l) \sigma_{(x\star z)}(h^{-1}l^{-1}\Gamma^t_z(h^{-1})) h^t(x ,z )) f^t(y\star z, x \star z) t(xy \star z). \tag*{(9)}
\end{align*} 
On comparing the equations $(8)$ and $(9)$, we get 
\begin{align*}
	l \Gamma^t_{xy}(l) f^t(x, y) ( \sigma_x(k) f^t(x, y) \star l) \sigma_{(xy) \star z}(\sigma_x(k^{-1}) {f^t(x, y)}^{-1}   \Gamma^t_{z}(\sigma_x(k^{-1}) \\& \hspace{-12.3cm} \cdot  f^t(x, y)^{-1}))   h^t(xy ,z) \\& \hspace{-12.7cm} = \sigma_x (l \Gamma^t_y(l) (k \star l) \sigma_{(y\star z)}(k^{-1}l^{-1}\Gamma^t_z(k^{-1}))h^t(y ,z)) f^t(x, y \star z) f^t(y \star z, x)^{-1}\\& \hspace{-12.3cm} \cdot \sigma_{(y\star z)} (l \Gamma^t_x(l)   h^t(x ,z ))f^t(y\star z, x \star z). \tag*{(10)}
\end{align*}
Next, consider the expression 
\begin{align*}
	ht(x)\star (kt(y)\cdot lt(z)) &=  ht(x)\star k\sigma^t_y(l) f^t(y, z) t(yz)  \\& \hspace{-3.6cm}=  hk\sigma_y(l) f^t(y, z) \Gamma^t_x(k\sigma_y(l) f^t(y, z)) (h \star k\sigma^t_y(l) f^t(y, z) ) \sigma_{(x \star yz)}(h^{-1}k^{-1}\sigma_y(l^{-1}) \\& \hspace{-3.2cm}\cdot f^t(y, z)^{-1} \Gamma^t_{yz}(h^{-1})) h(x ,yz) t(x \star yz).
	\tag*{(11)}
\end{align*}
On the other hand, we have
\begin{align*}
	ht(x)\star (kt(y)\cdot lt(z)) &= (ht(x)\star kt(y)) \cdot {^{kt(y)}(ht(x)\star lt(z))} \\&\hspace{-3cm}=
	hk \Gamma^t_x(k) (h \star k) \sigma_{(x\star y)}(h^{-1}k^{-1}  \Gamma^t_y(h^{-1}))h^t(x ,y )t(x \star y) kt(y) (hl \Gamma^t_x(l) (h \star l) \\& \hspace{-2.6cm}\cdot  \sigma_{(x\star z)}(h^{-1}l^{-1} \Gamma^t_z(h^{-1}))h^t(x ,z )t(x \star z)) t(y)^{-1}k^{-1}
	 \\&\hspace{-3cm}=
	hk \Gamma^t_x(k) (h \star k) \sigma_{(x\star y)}(h^{-1}k^{-1} \Gamma^t_y(h^{-1}))h^t(x ,y )t(x \star y) k \sigma_y(hl \Gamma^t_x(l) (h \star l) \\& \hspace{-2.6cm}\cdot  \sigma_{(x\star z)}(h^{-1}l^{-1}\Gamma^t_z(h^{-1}))h^t(x ,z ))t(y)t(x \star z)) t(y)^{-1}k^{-1}
	\\&\hspace{-3cm}= hk \Gamma^t_x(k) (h \star k)\sigma_{(x\star y)}(h^{-1}k^{-1}  \Gamma^t_y(h^{-1}))h^t(x ,y )t(x \star y) k \sigma_y(hl \Gamma^t_x(l)  \\& \hspace{-2.6cm}\cdot  (h \star l) \sigma_{(x\star z)}(h^{-1}l^{-1} \Gamma^t_z(h^{-1}))h^t(x ,z )) f^t(y, x \star z) f^t(y(x\star z), y^{-1}) \\& \hspace{-2.6cm} \cdot \sigma_{y(x \star z)}(f^t(y^{-1}, y)^{-1} \sigma_{y^{-1}}(k^{-1}))t(^y(x \star z)) \\&\hspace{-3cm}= hk \Gamma^t_x(k) (h \star k) \sigma_{(x\star y)}(h^{-1}k^{-1}  \Gamma^t_y(h^{-1}))h^t(x ,y )t(x \star y) k \sigma_y(hl \Gamma^t_x(l) \\& \hspace{-2.6cm}\cdot  (h \star l) \sigma_{(x\star z)}(h^{-1}l^{-1} \Gamma^t_z(h^{-1}))h^t(x ,z )) f^t(y, x \star z) f^t(y(x\star z), y^{-1}) \\& \hspace{-2.6cm} \cdot \sigma_{y(x \star z)}(f^t(y^{-1}, y)^{-1} \sigma_{y^{-1}}(k^{-1}))t(^y(x \star z)) \\&\hspace{-3cm}= hk \Gamma^t_x(k)h^t(x ,y ) (h \star k) \sigma_{(x\star y)}(h^{-1} \Gamma^t_y(h^{-1})   \sigma_y(hl \Gamma^t_x(l) (h \star l) \sigma_{(x\star z)}(h^{-1}l^{-1} \\& \hspace{-2.6cm}\cdot\Gamma^t_z(h^{-1}))h^t(x ,z )) f^t(y, x \star z) f^t(x \star z, y)^{-1}  \sigma_{(x \star z)}(k^{-1})) f^t(x\star y, x\star z)\\& \hspace{-2.6cm} \cdot t((x\star y)(x \star z)).  \tag*{(12)}
\end{align*}
From equations $(11)$ and $(12)$, we get
\begin{align*}
	\sigma_y(l) f^t(y, z) \Gamma^t_x(\sigma_y(l) f^t(y, z)) (h \star \sigma^t_y(l) f^t(y, z) ) \sigma_{(x \star yz)}(h^{-1}\sigma_y(l^{-1})\\& \hspace{-11.7cm}\cdot f^t(y, z)^{-1}  \Gamma^t_{yz}(h^{-1})) h^t(x ,yz) \\& \hspace{-12.1cm} = \sigma_{(x\star y)}(h^{-1}\Gamma^t_y(h^{-1})   \sigma_y(hl \Gamma^t_x(l) (h \star l) \sigma_{(x\star z)}(h^{-1}l^{-1}\Gamma^t_z(h^{-1}))h^t(x ,z ))\\& \hspace{-11.7cm}\cdot f^t(y, x \star z) f^t(x \star z, y)^{-1}  ) f^t(x\star y, x\star z) h^t(x ,y ). \tag*{(13)}
\end{align*}
Now, consider the expression
\begin{align*}
	^{lt(z)}(ht(x)\star kt(y)) &= lt(z) hk \Gamma^t_x(k) (h \star k) \sigma_{(x\star y)}(h^{-1}k^{-1}\Gamma^t_y(h^{-1}))h^t(x ,y )t(x \star y) t(z)^{-1}l^{-1}
	\\&\hspace{-2cm} =l\sigma_z (hk \Gamma^t_x(k) (h\star k) \sigma_{(x\star y)}(h^{-1}k^{-1} \Gamma^t_y(h^{-1}))h^t(x ,y )) f^t(z, x\star y) t(z(x\star y)) t(z^{-1}) \\&\hspace{-1.6cm} \cdot f^t(z, z^{-1})^{-1} l^{-1}
	  \\&\hspace{-2cm} =l\sigma_z (hk \Gamma^t_x(k) (h\star k) \sigma_{(x\star y)}(h^{-1}k^{-1} \Gamma^t_y(h^{-1}))h^t(x ,y )) f^t(z, x\star y) f^t(z(x\star y), z^{-1})  \\&\hspace{-1.6cm} \cdot t(x\star y) f^t(z, z^{-1})^{-1} l^{-1} 
	   \\&\hspace{-2cm} =l\sigma_z (hk \Gamma^t_x(k) (h\star k) \sigma_{(x\star y)}(h^{-1}k^{-1} \Gamma^t_y(h^{-1}))h^t(x ,y )) f^t(z, x\star y) f^t(z(x\star y), z^{-1})  \\&\hspace{-1.6cm} \cdot  \sigma_{(x\star y)} (f^t(z, z^{-1})^{-1} l^{-1}) t(x\star y) 
	    \\&\hspace{-2cm} =l\sigma_z (hk \Gamma^t_x(k) (h\star k) \sigma_{(x\star y)}(h^{-1}k^{-1} \Gamma^t_y(h^{-1}))h^t(x ,y )) f^t(z, x\star y) f^t( x\star y, z)^{-1}   \\&\hspace{-1.6cm} \cdot  \sigma_{(x\star y)} (l^{-1}) t(x\star y). \tag*{(14)}
\end{align*}
Also, we have
\begin{align*}
	^{lt(z)}ht(x)\star {^{lt(z)}} kt(y) &=  lt(z)ht(x)t(z)^{-1}l^{-1}\star {lt(z)}kt(y)t(z)^{-1}l^{-1}
	 \\&\hspace{-3.4cm} = l\sigma _z(h)f^t(z, x) t(zx) t(z^{-1})f^t(z, z^{-1})^{-1} l^{-1}\star l\sigma _z(k)f^t(z, y) t(zy) t(z^{-1}) f^t(z, z^{-1})^{-1} l^{-1}
	  \\&\hspace{-3.4cm} = l\sigma _z(h)f^t(z, x) f^t(zx, z^{-1}) \sigma_x (f^t(z, z^{-1})^{-1} l^{-1}) t(x)\star l\sigma _z(k)f^t(z, y) f^t(zy, z^{-1}) \\&\hspace{-3cm}\cdot  \sigma_y (f^t(z, z^{-1})^{-1} l^{-1}) t(y)
	  \\&\hspace{-3.4cm} = l\sigma _z(h)f^t(z, x) f^t(x, z)^{-1} \sigma_x ( l^{-1}) t(x)\star l\sigma _z(k)f^t(z, y) f^t(y, z)^{-1}   \sigma_y ( l^{-1}) t(y)
	\\&\hspace{-3.4cm} = l^2\sigma _z(hk)f^t(z, x) f^t(x, z)^{-1} f^t(z, y) f^t(y, z)^{-1}  \sigma_x(l^{-1}) \sigma_y(l^{-1}) \Gamma^t_x(l\sigma _z(k)f^t(z, y) f^t(y, z)^{-1} \\&\hspace{-3cm}\cdot \sigma_y(l^{-1})) \cdot (l\sigma _z(h) f^t(z, x) f^t(x, z)^{-1} \sigma_x(l^{-1})\star l\sigma _z(k)f^t(z, y) f^t(y, z)^{-1} \sigma_y(l^{-1})) \\&\hspace{-3cm} \cdot \sigma_{(x\star y)}(l^{-2}\sigma_z(h^{-1}k^{-1}) f^t(z, x)^{-1} f^t(x, z)f^t(z, y)^{-1} f^t(y, z) \sigma_x(l)\sigma_y(l)  \Gamma^t_y(l^{-1}\sigma _z(h^{-1})\\&\hspace{-3cm}\cdot \sigma_x(l) f^t(z, x)^{-1} f^t(x, z) )) h^t(x, y) t(x \star y). \tag*{(15)} 
\end{align*}
From equations $(14)$ and $(15)$, we have
\begin{align*}
	\sigma_z (\Gamma^t_x(k) (h\star k) \sigma_{(x\star y)}(h^{-1}k^{-1}\Gamma^t_y(h^{-1}))h^t(x ,y )) f^t(z, x\star y) f^t(x\star y, z)^{-1}  \\& \hspace{-12cm} = lf^t(z, x) f^t(x, z)^{-1} f^t(z, y) f^t(y, z)^{-1}  \sigma_{x}(l^{-1}) \sigma_{y}(l^{-1}) \Gamma^t_x(l\sigma _z(k)f^t(z, y) f^t(y, z)^{-1} \\&\hspace{-11.7cm}\cdot \sigma_y(l^{-1})) (l\sigma _z(h) f^t(z, x) f^t(x, z)^{-1} \sigma_x(l^{-1})\star l\sigma _z(k)f^t(z, y) f^t(y, z)^{-1} \sigma_y(l^{-1})) \\&\hspace{-11.7cm}\cdot \sigma_{(x\star y)}(l^{-1}\sigma_z(h^{-1}k^{-1}) f^t(z, x)^{-1} f^t(x, z)f^t(z, y)^{-1} f^t(y, z) \sigma_{x}(l) \sigma_{y}(l)  \Gamma^t_y(l^{-1}\sigma _z(h^{-1}) \\&\hspace{-11.7cm}\cdot \sigma_x(l) f^t(z, x)^{-1} f^t(x, z) )) h^t(x, y). \tag*{(16)}
\end{align*}
Next, we have
\begin{align*}
	hk^2 \Gamma^t_{x}(k) \sigma_y(l) \sigma_z(k^{-1}) \sigma_{(x\star y)}(h^{-1}k^{-1}\Gamma^t_{y}(h^{-1})) f^t(y, z) f^t(z, y)^{-1} h^t(x, y)
    \cdot  (hk (h\star k) \\& \hspace{-14.2cm} \sigma_{x\star y}(h^{-1}k^{-1}\Gamma^t_{y}(h^{-1})) h^t(x, y) \star k \sigma_y(l) \sigma_z(k^{-1}) f^t(y, z) f^t(z, y)^{-1} ) 
     \cdot  \Gamma^t_{(x\star y)}(k \sigma_y(l) \sigma_z(k^{-1}) \\& \hspace{-14.2cm} f^t(y, z) f^t(z, y)^{-1})
	  \cdot  \sigma_{((x\star y) \star z )}(h^{-1}k^{-2} \Gamma^t_{x}(k^{-1}) \sigma_y(l^{-1}) \sigma_z(k) \sigma_{(x\star y)}(hk\Gamma^t_{y}(h)) f^t(y, z)^{-1} \\& \hspace{-14.2cm} f^t(z, y)  h^t(x, y)^{-1}  \Gamma^t_{z}(h^{-1}k^{-1} \Gamma^t_{x}(k^{-1})(k\star h) \sigma_{(x\star y)}(hk\Gamma^t_{y}(h))  h^t(x, y)^{-1})) h^t(x\star y, z)
	 \\& \hspace{-14.2cm} \cdot \sigma_{((x\star y) \star z )}(	kl^2 \Gamma^t_{y}(l) \sigma_z(h) \sigma_x(l^{-1}) \sigma_{(y\star z)}(k^{-1}l^{-1}\Gamma^t_{z}(k^{-1})) f^t(z, x) f^t(x, z)^{-1} h^t(y, z) \cdot  (kl \\& \hspace{-14.2cm} (k\star l) \sigma_{y\star z}(k^{-1}l^{-1}\Gamma^t_{z}(k^{-1})) h^t(y, z) \star l \sigma_z(h) \sigma_x(l^{-1}) f^t(z, x) f^t(x, z)^{-1} ) 
	  \cdot  \Gamma^t_{(y\star z)}(l \sigma_z(h) \\& \hspace{-14.2cm} \sigma_x(l^{-1}) f^t(z, x) f^t(x, z)^{-1})	
	   \cdot  \sigma_{((y\star z) \star x )}(k^{-1}l^{-2} \Gamma^t_{y}(l^{-1}) \sigma_z(h^{-1}) \sigma_x(l) \sigma_{(y\star z)}(kl\Gamma^t_{z}(k)) \\& \hspace{-14.2cm} f^t(z, x)^{-1} f^t(x, z) h^t(y, z)^{-1} \Gamma^t_{x}(k^{-1}l^{-1} \Gamma^t_{y}(l^{-1})(l\star k) \sigma_{(y\star z)}(kl\Gamma^t_{z}(k))  h^t(y, z)^{-1})) \\& \hspace{-14.2cm} h^t(y\star z, x)) \cdot  f^t((x\star y)\star z, (y\star z) \star x ) \cdot \sigma_{((x\star y)\star z)((y\star z)\star x)} (
	 	lh^2 \Gamma^t_{z}(h) \sigma_x(k)  \sigma_y(h^{-1})\\& \hspace{-14.2cm} \sigma_{(z\star x)}(l^{-1}h^{-1}\Gamma^t_{x}(l^{-1})) f^t(x, y) f^t(y, x)^{-1} h^t(z, x)
	  \cdot (lh (l\star h) \sigma_{z\star x}(l^{-1}h^{-1}\Gamma^t_{x}(l^{-1})) h^t(z, x) \\& \hspace{-14.2cm} \star h  \sigma_x(k) \sigma_y(h^{-1}) f^t(x, y) f^t(y, x)^{-1} ) \cdot \Gamma^t_{(z\star x)}(h \sigma_x(k) \sigma_y(h^{-1}) f^t(x, y) f^t(y, x)^{-1})
	   \\& \hspace{-14.2cm}  \cdot \sigma_{((z\star x) \star y )}(l^{-1}h^{-2} \Gamma^t_{z}(h^{-1}) \sigma_x(k^{-1}) \sigma_y(h) \sigma_{(z\star x)}(lh\Gamma^t_{x}(l)) f^t(x, y)^{-1} f^t(y, x) h^t(z, x)^{-1} \\& \hspace{-14.2cm} \Gamma^t_{y}(l^{-1}h^{-1} \Gamma^t_{z}(h^{-1})(h\star l) \sigma_{(z\star x)}(lh\Gamma^t_{x}(l))  h^t(z, x)^{-1})) h^t(z\star x, y))
	    \cdot  f^t(((x\star y)\star z) \\& \hspace{-14.2cm}( (y\star z)\star x), (z\star x)\star y) = 1.\tag*{(17)}	 
\end{align*}

On the basis of above discussion, we have the following proposition: 

\begin{proposition}	Let $G$ be a nilpotent group of class $2$ equipped with a  multiplicative Lie algebra structure $\star$ such that the induced multiplicative Lie algebra structure on $\frac{G}{[G,G]}$ is trivial. Then $\star$ is determined by a  map $h : \frac{G}{[G,G]} \times \frac{G}{[G,G]} \to [G,G]$ which satisfies the following conditions for all $x, y, z\in \frac{G}{[G,G]}:$
	\begin{enumerate}
		\item $h(x,x)=1  $;
		
		\item $  h(xy ,z) =  h(x ,z )h(y ,z )  $;
		
		\item $   h(x ,yz) = h(x ,y )h(x ,z ) $.
		
	\end{enumerate}
%	Then, we have a multiplicative Lie algebra structure $\star$ on $G$ defined  by 
%	\begin{align*}
%		ht(x)\star kt(y) =  h^t(x,y), \ \text{for all} \ ht(x), kt(y) \in G.
%	\end{align*} 
\end{proposition}

\begin{proof}
	From Lemma \ref{LZG}, we have $[G,G]\subseteq Z(G)\cap LZ(G)$.  Thus, we have the following central extension of multiplicative Lie algebras:
	$$ \mathcal{L}:	{1}\longrightarrow [G,G]\overset{i} \longrightarrow G\overset{\pi} \longrightarrow  \frac{G}{[G,G]} \longrightarrow {1}.$$ Let $t$ be a section of $ \mathcal{L}$. Then every element of $G$ can be written in the form of $ht(x)$ for some  $h \in [G,G]$ and $x\in \frac{G}{[G,G]}$. Now, by above discussion, we have maps $\sigma: \frac{G}{[G,G]} \to Aut([G,G]), \ \Gamma^t :\frac{G}{[G,G]} \to End([G,G]), \  f^t: \frac{G}{[G,G]} \times \frac{G}{[G,G]} \to [G,G],$ and $ h^t: \frac{G}{[G,G]} \times \frac{G}{[G,G]} \to [G,G]$ satisfying the equations $(2), (3), (6), (10), (13), (16) $ and $(17).$ Since $[G,G]\subseteq Z(G)$, $\sigma_x(h) = t(x)ht(x)^{-1} = h$ for all $h \in [G,G]$ and $x\in \frac{G}{[G,G]}$.
	
	Moreover, since $[G,G]\subseteq LZ(G)$, we have  $\Gamma^t_x(k)=t(x)\star k = 1$ for all $k \in [G,G]$ and $x\in \frac{G}{[G,G]}$.
	Hence, we have 
		\begin{align*}
		ht(x)\star kt(y) =  h^t(x,y).
		\end{align*} 
Now, it suffices to show that the map  $h^t$ is independent to the choice of section $t$. Let $s$ be another section of $ \mathcal{L}$. Then there is a map $g:\frac{G}{[G,G]}\to [G,G]$	such that $s(x)=g(x)t(x)$ for each $x$, and then $h^s(x,y)=g(x\star y)^{-1}h^t(x,y)=h^t(x,y)$.  
\end{proof}

\begin{theorem} \label{theorem}
	Let	$H$ and $K$  be two Lie rings with Lie ring structure $\star_1$ and $\star_2$, respectively. Suppose we have maps $\sigma: K \to Aut(H), \ \Gamma : K \to End(H), \  f: K \times K \to H,$ and $ h: K \times K \to H$ satisfying the equations $(2), (3), (6), (10), (13), (16) $ and $(17).$ 
	
Define the following two operations on $G= H \times K$
	\begin{align*} 
	(h,x)\cdot (k,y) &= (h\sigma_x(k)f(x, y), xy)\\
		(h,x){\star} (k,y) &= (hk \Gamma_x(k)(h \star_1 k) \sigma_{(x \star_2 y)} (h^{-1}k^{-1} \Gamma_y(h^{-1})) h(x, y), x \star_2 y),
	\end{align*}
	$ \text{for all} \ (h,x), (k,y) \in G.$ Then $(G, \cdot, \star)$ is a multiplicative Lie algebra.
\end{theorem}

A group $G$ is metacyclic if it contains a cyclic normal subgroup $N$ such that $\frac{G}{N}$
is also cyclic. 

\begin{remark}
\begin{enumerate}
\item Let $G$ be a metacyclic group equipped with a multiplicative Lie algebra structure $\star$. Suppose $H$ is  a cyclic normal subgroup  such that $\frac{G}{H}$ is also cyclic. Then, we have the following extension of multiplicative Lie algebras:
$$ \mathcal{E}:	{1}\longrightarrow  H\overset{i} \longrightarrow G\overset{\beta} \longrightarrow \frac{G}{H} \longrightarrow {1}.$$ 
Let $t$ be a section of $ \mathcal{E}$. Then we have maps $\sigma: \frac{G}{H} \to Aut(H), \ \Gamma^t : \frac{G}{H} \to End(H), \  f^t: \frac{G}{H} \times \frac{G}{H} \to H,$ and $ h^t: \frac{G}{H} \times \frac{G}{H} \to H$ satisfy the equations $(2), (3) $ and the following conditions for all $x, y, z\in \frac{G}{H}:$ 
\begin{enumerate}
	\item $    \Gamma^t_{z}(\sigma_x(k^{-1})  f^t(x, y)^{-1}))  h^t(xy ,z)  =  \sigma_x (  \Gamma^t_z(k^{-1})  h^t(y ,z)) h^t(x ,z )$;\vspace{0.1cm}
	
	\item $ \Gamma^t_x(\sigma_y(l) f^t(y, z))  h^t(x ,yz)  =   \sigma_y( \Gamma^t_x(l) h^t(x ,z )) h^t(x ,y )$;\vspace{0.1cm}
	
	\item $  	\sigma_z (  \Gamma^t_x(k) \Gamma^t_y(h^{-1})h^t(x ,y ))    =   \Gamma^t_x(l\sigma _z(k)f^t(z, y) f^t(y, z)^{-1}  \sigma_y(l^{-1})) \\  \cdot \Gamma^t_y(l^{-1}\sigma _z(h^{-1})\sigma_x(l) f^t(z, x)^{-1} f^t(x, z) ) h^t(x, y)$;\vspace{0.1cm}
	
	\item  $  \Gamma^t_{z}( \Gamma^t_x(k^{-1}) \Gamma^t_y(h) h^t(x ,y )^{-1})      \Gamma^t_{x}( \Gamma^t_y(l^{-1})\Gamma^t_z(k) h^t(y ,z )^{-1}))   \Gamma^t_{y}( \Gamma^t_z(h^{-1}) \Gamma^t_x(l)\\ h^t(z ,x )^{-1}))  = 1.$
	
\end{enumerate}

Since $H$ and $\frac{G}{H}$ have trivial multiplicative Lie algebra structure, $\star$ is  given as follows:  
$$ht(x)\star kt(y) =   \Gamma^t_x(k) \Gamma^t_y(h^{-1}) h^t(x, y), \ \text{for all} \ ht(x), kt(y) \in G.$$ 

\item Let $G$ be a group of order $p^n$ with two generators and nilpotency class two. Then $G$ may be viewed as an extension:
$$ \mathcal{F}:	{1}\longrightarrow  C_{p^a}\overset{i} \longrightarrow G\overset{\beta} \longrightarrow C_{p^b}\times C_{p^c} \longrightarrow {1},$$
where $C_{p^a}\subseteq Z(G), n=a+b+c$ and $a\geq b\geq c\geq 1$ (See \cite{AAR}). Suppose $\star$ is a multiplicative Lie algebra structure on $G$ such that the quotient multiplicative Lie algebra $\frac{G}{C_{p^a}} \cong C_{p^b}\times C_{p^c}$ equipped with trivial multiplicative Lie algebra structure. Then $\mathcal{F}$ is an extension of multiplicative Lie algebras. Let $t$ be a section of $ \mathcal{F}$. Then we have maps $\sigma: C_{p^b}\times C_{p^c}  \to Aut(C_{p^a}), \ \Gamma^t : C_{p^b}\times C_{p^c} \to End(C_{p^a}), \  f^t:  C_{p^b}\times C_{p^c} \times  C_{p^b}\times C_{p^c} \to C_{p^a},$ and $ h^t:  C_{p^b}\times C_{p^c} \times  C_{p^b}\times C_{p^c} \to C_{p^a}$ satisfies the equations $(2), (3) $ and following conditions for all $x, y, z\in  C_{p^b}\times C_{p^c}:$ 
\begin{enumerate}
	\item $    \Gamma^t_{z}(\sigma_x(k^{-1})  f^t(x, y)^{-1}))  h^t(xy ,z)  =  \sigma_x (  \Gamma^t_z(k^{-1})  h^t(y ,z)) h^t(x ,z )$;\vspace{0.1cm}
	
	\item $ \Gamma^t_x(\sigma_y(l) f^t(y, z))  h^t(x ,yz)  =   \sigma_y( \Gamma^t_x(l) h^t(x ,z )) h^t(x ,y )$;\vspace{0.1cm}
	
	\item $  	\sigma_z (  \Gamma^t_x(k) \Gamma^t_y(h^{-1})h^t(x ,y ))    =   \Gamma^t_x(l\sigma _z(k)f^t(z, y) f^t(y, z)^{-1}  \sigma_y(l^{-1})) \\  \cdot \Gamma^t_y(l^{-1}\sigma _z(h^{-1})\sigma_x(l) f^t(z, x)^{-1} f^t(x, z) ) h^t(x, y)$;\vspace{0.1cm}
	
		\item  $  \Gamma^t_{z}( \Gamma^t_x(k^{-1}) \Gamma^t_y(h) h^t(x ,y )^{-1})      \Gamma^t_{x}( \Gamma^t_y(l^{-1})\Gamma^t_z(k) h^t(y ,z )^{-1}))   \Gamma^t_{y}( \Gamma^t_z(h^{-1}) \Gamma^t_x(l)\\ h^t(z ,x )^{-1}))  = 1.$
	
\end{enumerate}

Thus $\star$ is given as follows: 
$$ht(x)\star kt(y) =   \Gamma^t_x(k) \Gamma^t_y(h^{-1}) h^t(x, y)\ \text{for all} \ ht(x), kt(y) \in G. $$ 

\end{enumerate}
\end{remark}

\noindent{\bf Acknowledgement:} We are thankful to Dr. Seema Kushwaha for her constant support. The first named author sincerely thanks IIIT Allahabad and Ministry of Education, Government of India for providing institute fellowship. The second named author sincerely thanks IIIT Allahabad and University grant commission (UGC), Govt. of India, New Delhi for research fellowship. The third named author is thankful to National Board for Higher Mathematics (NBHM), Government of India  for the financial support for the project ``Linear Representation of multiplicative Lie algebra".

\end{document}